\documentclass[a4paper,11pt,twoside]{article}

\usepackage[utf8]{inputenc}
\usepackage[T1]{fontenc}
\usepackage[english]{babel}

\usepackage{charter}

\usepackage{indentfirst}
\usepackage{xcolor}
\usepackage{verbatim}

\usepackage{hyperref}

\usepackage{pifont}
\newcommand\dd{\text{d}}

\usepackage[top=3.cm, bottom=4.0cm, left=2.2cm, right=2.2cm]{geometry}
\usepackage{amsmath}
\usepackage{amsthm}	
\usepackage{amsfonts}	
\usepackage{amssymb	
           ,bbm
            ,units 
           }
\usepackage{enumerate}
\usepackage[nobysame
           ,alphabetic
           ,numeric
           ,initials
           ]{amsrefs}

\usepackage{
            ulem		
           ,soul		
} 

\numberwithin{equation}{section}
\numberwithin{figure}{section}
 \usepackage[nodayofweek]{datetime}

\renewcommand*{\thefootnote}{\fnsymbol{footnote}}

\title{On the relation between Stratonovich and It\^o integrals with functional integrands of conditional measure flows}
\author{
 	\normalsize Gon\c calo dos Reis\footnote{G. dos Reis acknowledges support from the \emph{Funda{\c c}$\tilde{\text{a}}$o para a Ci$\hat{e}$ncia e a Tecnologia} (Portuguese Foundation for Science and Technology) through the project [UIDB/00297/2020] (Centro de Matem\'atica e Aplica\c c$\tilde{\text{o}}$es CMA/FCT/UNL).}
 	\\[8pt]
         \small  University of Edinburgh\\ 
         \small  School of Mathematics \\
         \small  Edinburgh, EH9 3FD, UK\\  
         \small  and \\
	\small  Centro de Matem\'atica e Aplica\c c$\tilde{\text{o}}$es \\
	\small (CMA), FCT, UNL, Portugal \\
        \small  G.dosReis@ed.ac.uk
 \and
\normalsize Vadim Platonov\\[8pt] 
         \small  University of Edinburgh\\ 
         \small  School of Mathematics \\
         \small  Edinburgh, EH9 3FD, UK\\    
        \small    \\
		\small  \\
          \\
        \small V.d.Platonov@sms.ed.ac.uk
}

\date{ \currenttime, \ddmmyyyydate\today\qquad{(File: \tt \jobname.tex})}

\theoremstyle{plain}
\newtheorem{theorem}{Theorem}[section]

\newtheorem{proposition}[theorem]{Proposition}

\newtheorem{definition}[theorem]{Definition}

\newtheorem{remark}[theorem]{Remark}

\newtheorem{assumption}[theorem]{Assumption}

\newcommand{\bE}{\mathbb{E}}
\newcommand{\bF}{\mathbb{F}}

\newcommand{\bN}{\mathbb{N}}
\newcommand{\bP}{\mathbb{P}}

\newcommand{\bR}{\mathbb{R}}

\newcommand{\cB}{\mathcal{B}}
\newcommand{\cC}{\mathcal{C}}

\newcommand{\cF}{\mathcal{F}}

\newcommand{\cP}{\mathcal{P}}

\newcommand{\trace}{\textrm{Trace}}
\newcommand{\Supp}{\textrm{Supp}}
\newcommand{\Law}{\textrm{Law}}

\definecolor{darkgreen}{rgb}{0,0.35,0}

\newcommand{\1}{\mathbbm{1}}

\begin{document}

\selectlanguage{english}

\maketitle
\renewcommand*{\thefootnote}{\arabic{footnote}}

\vspace{-1cm}
\begin{abstract} 
In this small note we explicit the relation between It\^o and Stratonovich integrals when conditional measure flow components are present in the integrands. The `correction' term involves Lions-type measure derivatives and clarifies which cross-correlations need to be taken into account. We cast the framework in relation to SDEs of mean-field type depending on conditional flows of measure. The result being trivial under full flows of measure.
\end{abstract}
{\bf Keywords:} 
Stratonovich integral, It\^o integral, Stochastic Differential Equations, conditional measure flows, Lions derivative, mean-field
\vspace{0.3cm}

\noindent
{\bf 2010 AMS subject classifications:}\\
Primary: 60H05 
Secondary: 60H10, 60H15

%

%
%
%
\footnotesize
\setcounter{tocdepth}{2}
\tableofcontents
\normalsize


\newpage
\section{Introduction}

It\^o and Stratonovich integrals are ubiquitous in stochastic analysis \cite{KaratzasShreve1991BMandStochCalc,kunita1997stochastic,kloedenplaten1992NumericalSDEs} and each approach with its own advantages and disadvantages depending on usage \cite{vankampen1981itovsStrat,holm2020itovsStrat}. For instance, while the Stratonovich form has a simple chain rule, its drawback lies in the fact that expectations of Stratonovich integrals are more difficult to control (their expectation is non-zero in strict opposition to the It\^o integral). Examples of this \textit{It\^o Vs.~Stratonovich} dichotomy can be found in the motivation for numerical methods \cite{Foster2020polynomialapproximation}, regularization by noise \cite{Maurelli2020Nonexplosion,Kupferman2004ItoStrat} or the recent development of probabilistic rough paths \cite{delarue2021probabilistic}: each approach leads to different forms of iterated integrals (in Taylor expansions) and, depending on the context, one is more convenient to use than the other. 

Recall the context of Stochastic Differential Equations (SDE) with $W$ denoting a standard Brownian motion and, for sake of argument, let $b$ and $\sigma$ be `nice' maps. Denote the Stratonovich integral with its usual symbol  $\circ \dd W$ with  $\dd W$ referring to the usual It\^o integral.  

A key result in stochastic analysis is the conversion rule between the  It\^o and Stratonovich integrals in the context of SDEs, here in one dimension, namely 
\begin{align*}
\mathrm{d} X_t=b(t,X_t) \mathrm{d} t+\sigma(t,X_t) \circ \mathrm{d} W_t 
\ \ \Leftrightarrow \ \ 
\mathrm{d} X_t=\big(\,b(t,X_t)+\frac{1}{2} \sigma(t,X_t) \partial_{x} \sigma(t,X_t)\,\big) \mathrm{d} s+\sigma(t,X_t) \mathrm{d} W_t,
\end{align*}
and, conversely, that
\begin{align*}
\mathrm{d} X_t=b(t,X_t) \mathrm{d} t+\sigma(t,X_t) \mathrm{d} W_t 
\ \ \Leftrightarrow \ \ 
\mathrm{d} X_t=\big(\,b(t,X_t)-\frac{1}{2} \sigma(t,X_t) \partial_{x} \sigma(t,X_t)\,\big) \mathrm{d} t+\sigma(t,X_t) \circ \mathrm{d} W_t.
\end{align*}

The moral behind the correction term $\sigma \partial_x \sigma$ is the same as in the It\^o-Wentzell formula \cite{kunita1997stochastic,Platonov2019ito}: it emerges from the  cross variation  between the randomness in the integrand $\sigma(t,X_t)$ and the driving noise when one moves from the Stratonovich mid-point integration rule to the It\^o left-point integration rule (and vice-versa). If $\sigma$ has no random elements (or existing, they are independent of the driving noise) then the It\^o and Stratonovich integrals coincide. 
\smallskip

The goal of this work is to clarify the relation between It\^o and Stratonovich integrals when conditional measure-flow components, in the style of mean-field games, are present in the integrands \cite{Platonov2019ito,CarmonaDelarue2017book2}. Namely, when $\sigma$ is not just a function of time and space but also a measure functional and the measure input is a \textit{conditional measure flow} -- if the measure is not random then the problem is trivial as explained above. Our contribution is motivated by developments in mean-field games with common noise \cite{CarmonaDelarue2017book2,delarue2021probabilistic,cardaliaguet2015master}, mean-field SDEs and conditional propagation of chaos \cite{Erny2021ConditoinalPoC}, open problems in regularization by noise (for conditional measures flows) or numerical methods among others. The contributions within the framework of conditional measure flows is slowly developing and these results are a step in supporting this literature. 

In terms of the results themselves there are two comments to be made. Firstly, the `correction' term involves a Lions-type measure derivative \cite{CarmonaDelarue2017book1,CarmonaDelarue2017book2} in a way that is intuitively comparable to the $\sigma \partial_x \sigma$ term appearing above. The result is intuitive but critically clarifies that no cross noise variation appears between the (possible) multiple sources of randomness in the integrand. 
The second comment relates to the employed methodology: we follow is the ``limit of particles systems'' one. This approach was initiated in \cite{chassagneux2014classical,CarmonaDelarue2017book1,CarmonaDelarue2017book2} and employed in \cite{Platonov2019ito} to prove an It\^o-Wentzell-Lions formula. Our result follows the arguments in the latter. Critically, we point the reader to the introduction of \cite{Platonov2019ito} for a review on alternative proof methodologies \cite{buckdahn2017mean,cardaliaguet2015master,cavallazzi2021krylov,guo2020s,talbi2021dynamic}. On the latter two, \cite{guo2020s,talbi2021dynamic}, the approach for conditional flows is not developed and this is the critical case for this manuscript. Nonetheless, we believe their approach, under the assumption that the diffusion map $\sigma$ is deterministic, would deliver the same results under the same assumptions. For simplicity we follow \cite{Platonov2019ito}.

\textbf{Organisation of the paper.} In Section 2 we set notation and review a few concepts necessary for the main constructions. The main result is presented in Section \ref{sec:conditionalflow}. 
\medskip

\textbf{Acknowledgements.} The authors would like to thank William Salkeld (Universit\'e de Nice Sophia-Antipolis, FR) for the helpful discussions.

%
%
%

%
%
%
\section{Notation and auxiliary results}
\label{sec:two}

\subsection{Notation and Spaces}

Let $\bN$ be the set of natural numbers starting at $1$, $\bR$ denotes the real numbers. For  collections of vectors in $\{x^l\}_{l}\in \bR^d$, let the upper index $l$ denotes the distinct vectors, whereas the lower index the vector components, i.e. $x^l=(x^l_1,\cdots,x^l_d)\in \bR^d$ namely $x^l_j$ denotes the $j$-th component of $l$-th vector. For $x,y \in \bR^d$ denote the scalar product by $x \cdot y=\sum_{j=1}^d x_j y_j$; and $|x|=(\sum_{j=1}^d x_j^2)^{1/2}$ the usual Euclidean distance; and $x \otimes y$ denotes the tensor product of vectors $x,y \in \bR^d$. Let $\1_A$ be the indicator function of set $A\subset \bR^d$. For a matrix $A \in \bR^{d\times m}$ we denote by $A^\intercal$ its transpose and its Frobenius norm by $|A|=\trace\{A A^\intercal\}^{1/2}$. Let $I_d:\bR^d\to \bR^d$ be the identity map. 

We denote by $\cC(A,B)$ for $A,B \subseteq \bR^d$, $d\in \bN$, the space of continuous functions $f:A\to B$.  In terms of derivative operators and differentiable functions, $\partial_t$ denotes the partial differential in the time parameter $t \in [0,T]$; $\partial_y$ and $\partial_{yy}^2$ denote the gradient and the Hessian operator in $y\in\bR^d$ respectively. 

We say that the function is locally bounded, when its restriction to the compact set is bounded.

\subsection*{Spaces}

We introduce over $\bR^d$ the space of probability measures $\cP(\bR^d)$ and its subset $\cP_2(\bR^d)$ of those with finite second moment. The space $\cP_2(\bR^d)$ is Polish under the Wasserstein distance
\begin{align*}
W_2(\mu,\nu) = \inf_{\pi\in\Pi(\mu,\nu)} \Big(\int_{\bR^d\times \bR^d} |x-y|^2\pi(dx,dy)\Big)^\frac12, \quad \mu,\nu\in \cP_2(\bR^d) ,
\end{align*}   
where $\Pi(\mu,\nu)$ is the set of couplings for $\mu$ and $\nu$ such that $\pi\in\Pi(\mu,\nu)$ is a probability measure on $\bR^d\times \bR^d$ such that $\pi(\cdot\times \bR^d)=\mu$ and $\pi(\bR^d \times \cdot)=\nu$. Let $\Supp(\mu)$ denote the support of $\mu \in \cP(\bR^d)$.

Throughout set some $0<T<+\infty$ and we work the finite time interval $[0,T]$. We consider $(\Omega^0,\cF^0, \bF^0 = (\cF^0_t)_{t \in [0,T]},\bP^0)$ and $(\Omega^1,\cF^1, \bF^1 = (\cF^1_t)_{t \in [0,T]},\bP^1)$ atomless Polish probability spaces to be the respective completions of $(\Omega^0,\bF^0,\bP^0)$ and $(\Omega^1,\bF^1,\bP^1)$ carrying a respective $m$-dimensional Brownian motions $W^0 = (W_t^0)_{t \in [0,T]}$ and $W^1 = (W_t^1)_{t \in [0,T]}$ generating the probability space's filtration, augmented by all $\bP^0$- and $\bP^1$-null sets respectively. We augment $(\Omega^0,\cF^0, \bF^0 = (\cF^0_t)_{t \in [0,T]},\bP^0)$ with a sufficiently rich sub $\sigma$-algebra $\cF^0_0$ independent of $W^0$ and $W^1$. We denote by $(\Omega,\bF, \bP)$ the completion of the product space $(\Omega^0 \times \Omega^1,\bF^0 \otimes \bF^1, \bP^0 \otimes\bP^1)$ equipped with the filtration $\bF$ obtained by augmenting the product filtration $\bF^0 \otimes \bF^1$ in a right-continuous way and by completing it. We let $\bE^0$ and $\bE^1$ taking the expectation on the first and second space respectively. 
We adopt the following convention, that for $d$-dimensional random vector $Y=(Y_1,\cdots,Y_d)$ we understand denote $\bE[Y]$ by the $d$-dimensional vector $(\bE[Y_1],\cdots,\bE[Y_d])$.

We define $L^2(\Omega,\cF_0,\bP,\bR^d)$ as the space of $\cF_0$-measurable random variables $\xi:\Omega\to \bR^d$ that are square integrable $\bE^\bP[|\xi|^2]<\infty$. Given two processes $(X_t)_{t \in [0,T]}$ and $(Y_t)_{t \in [0,T]}$ let $\langle X_t,Y_t \rangle$ denote their cross-variation up to time $t \in [0,T]$.

Lastly, for convenience we choose to work over $1$-, $d$- and $d\times m$-dimensional spaces. 
The generalisation to different dimensions is straightforward from our text.

\subsection{The Lions derivative}

\subsubsection{The Lions derivative and notational conventions}
\label{sec:LionsDerivative}

To consider the calculus for the mean-field setting one requires to build a suitable differentiation operator on the $2$-Wasserstein space. Among the several notions of differentiability of a functional $u$ defined over  $\cP_2(\bR^d)$ we follow the approach introduced by Lions in his lectures at Coll\`ege de France \cite{lions2007cours} and further developed in \cite{cardaliaguet2010notes}. A comprehensive presentation can be found in the joint monograph of Carmona and Delarue \cite{CarmonaDelarue2017book1},\cite{CarmonaDelarue2017book2}.

We consider a canonical lifting of the function $u:\cP_2(\bR^d) \to \bR$ to $\tilde{u}: L^2(\Omega,\cF,\bP;\bR^d) \ni Y \to \tilde u (Y) = u(Law(Y)) \in \bR$, where  $L^2(\Omega,\cF,\bP;\bR^d)$ is a space of square integrable random variables. We can say that $u$ is $L$-differentiable at $\mu$, if $\tilde u$ is Frech\`et differentiable (in $L^2$) at some $Y$, such that $\mu = \bP \circ Y^{(-1)}$. Denoting the gradient by $D\tilde u$ and using a Hilbert structure of the $L^2$ space, we can identify $D\tilde u$ as an element its dual, $L^2$ itself. It was shown in \cite{cardaliaguet2010notes} that $D\tilde u$ is a $\sigma(Y)$-measurable random variable and given by the function $Du(\mu, \cdot) : \bR^d \to \bR^d $, depending on the law of $Y$ and satisfying $Du (\mu, \cdot) \in L^2(\bR^d, \cB(\bR^d),\mu; \bR^d)$. Hereinafter the  $L$-derivative of $u$ at $\mu$ is the map $\partial_\mu u(\mu,\cdot): \bR^d \ni v \to \partial_\mu u(\mu,v) \in \bR^d$, satisfying $D\tilde u(Y) = \partial_\mu u(\mu,Y)$.
We always denote $\partial_\mu u$ as the version of the $L$-derivative that is continuous in the product topology of all components of $u$.

While working with a common noise setting, we reconstruct the space by adding the copy spaces, this procedure is well-described in the proof of Proposition \ref{prop:strat-to-ito-main-result}.

\subsubsection{The Empirical projection map}
We recall the concept of \emph{empirical projection map} given in \cite{chassagneux2014classical} which will be one of the main workhorses throughout our work.
\begin{definition}[Empirical projection of a map]
\label{def:Auxiliary-uN-for-EmpirialTrick}
Given $u: \cP_2(\bR^d) \to \bR$ and $N\in \bN$, define the empirical projection $u^N$ of $u$ via $u^N: (\bR^d)^N \to \bR$, such that \[u^N(y^1,\dots, y^N) := u \big(\bar{\mu}^N\big),
\quad \text{with}\quad
\bar{\mu}^N := \frac{1}{N}\sum\limits_{l=1}^N \delta_{y^l}
\quad\textrm{and}\quad 
y^l\in \bR^d ,~ l=1,\dots,N.
\]
\end{definition}
We recall \cite{CarmonaDelarue2017book1}*{Proposition 5.35} which relates the spatial derivative of $u^N$ with the $L$-derivative of $u$.
\begin{proposition}
\label{prop:DerivativeRelations-Space-2-Lions}
Let $u: \cP_2(\bR^d) \to \bR$ be \textit{Fully}-$\cC^2 (\cP_2(\bR^d))$, then, for any $N>1$, the empirical projection $u^N$ is $\cC^2$ on $(\bR^d)^N$ and for all $y^1,\cdots,y^N\in\bR^d$ we have the following differentiation rule
\begin{align*}
    \partial_{y^j}u^N(y^1, \dots, y^N) &= \frac{1}{N}\: \partial_\mu u\Big(\frac{1}{N}\sum_{l=1}^N \delta_{y^l},y^j\Big).
\end{align*}
\end{proposition}


%
%
%
%

\section{Stratonovich - It\^o equivalence under conditional flows of measures}
\label{sec:conditionalflow}

The setting discussed in this section is inspired by the developments in the theory of mean-field games with common noise, \cite{cardaliaguet2015master} and \cite{CarmonaDelarue2017book2}.

Take measurable $(b,\sigma^0, \sigma^1): [0,T] \times \cP_2(\bR^d) \to \bR^d \times \bR^{d \times m }\times \bR^{d\times m}$ and define the following process
\begin{align}
    \label{eq:GenericYprocess-PartialFlow22}
    \dd Y_t = b(t,\mu_t) \dd t + \sigma^0(t,\mu_t) \dd W^0_t + \sigma^1(t,\mu_t) \dd W^1_t, \text{ and initial condition } Y_0\in L^2(\Omega,\cF_0,\bP),
\end{align}
and $\mu_t := \Law(Y_t(\omega_0,\cdot))$ for $\bP^0$-almost any $\omega_0$.
Here $\Law(Y_t(\omega_0,\cdot))$ can be understood as RV from $(\Omega^0,\cF^0,\bP^0)$ into $\cP(\bR^d)$ (for further details see discussion in \cite{CarmonaDelarue2017book2}*{Section 4.3}).

The components of $Y = (Y_1,\dots, Y_d) \in \bR^d$, satisfy the following 1-dimensional SDE  
\begin{align}
    \label{eq:GenericYprocess-PartialFlow22-1d}
    \dd Y_{i,t} = b_i(t,\mu_t) \dd t + \sum_{j = 1}^m \sigma_{ij}^0 (t,\mu_t) \dd W^0_{j,t} +
    \sum_{j = 1}^m \sigma_{ij}^1(t,\mu_t) \dd W^1_{j,t}, \:\: Y_{0,i}\in L^2(\Omega,\cF_0,\bP), 
\end{align}
for all $i = 1,\dots, d$.

Moreover, the involved coefficients will be assumed to satisfy the next condition.
\begin{assumption}
\label{Assump:SDE-Y-mu-2BM}
Let $Y_0 \in L^2(\Omega,\cF_0,\bP)$ ($Y_0$ is $\cF_0$-measurable and independent of $W_t^0,~W_t^1, ~ t \in [0,T]$). Take $b: [0,T] \times \cP_2(\bR^d) \to \bR^d$ and $\sigma^0,~\sigma^1:[0,T] \times \cP_2(\bR^d)\to\bR^{d \times m}$ such that $(b_t)_{t \in [0,T]}, (\sigma^0_t)_{t \in [0,T]}$ and $(\sigma^1_t)_{t \in [0,T]}$ are $\bF$-progressively measurable processes and satisfy
\begin{enumerate}[i)]
    \item For any $t \in [0,T]$, the maps $\mu \mapsto \sigma^0(t,\mu),~ \mu \mapsto \sigma^1(t,\mu)$ are continuous in topology induced by the Wasserstein metric for any $\mu \in \cP_2(\bR^d)$;
    
    \item For any $t \in [0,T],~i = 1,\dots, d$ and $j = 1,\dots, m$, the map $\mu \mapsto \sigma^0_{ij}(t,\mu)$ continuously L-differentiable at every point $\mu \in \cP_2(\bR^d)$. Moreover, for all $i = 1,\dots, d$ and $j = 1,\dots, m,~\partial_\mu \sigma_{ij}^0: [0,T] \times \cP_2(\bR^d)$ is joint-continuous and locally bounded at every triple $(t,\mu,v)$, with $(t,\mu) \in [0,T] \times \cP_2(\bR^d),~ v \in \Supp(\mu)$;

    \item The involved coefficients satisfy for any compact $K \subset \cP_2(\bR^d)$
        \begin{align}
            \label{cond:integrability-measure-only}
            \sup_{\mu \in K}\sup_{i \in \{1,\dots,d \}} \int_0^t\Big[ |b_i(s,\mu)| + |\sigma_i^0(s,\mu)|^2 + |\sigma_i^1(s,\mu)|^2 + \big|\int_{\bR^d}\partial_\mu \sigma_i^0(s,\mu,y) \mu(\dd y) \big|^2 \Big] \dd s  < \infty.
        \end{align}
\end{enumerate}
\end{assumption}

We name $(W_t^0)_{t \in [0,T]}$ as a common noise affecting the whole setting, whilst $(W^1)_{t \in [0,T]}$ is the idiosyncratic chaos for the process $Y$.
For the purposes of the present section we fix the common noise and make all the transforms by conditioning on $W^0$. The measurability of involved measure-derivative terms is discussed in \cite{CarmonaDelarue2017book1}*{Remarks 5.101 and 5.103}. Throughout the text we exploit the procedure of conditioning on the common noise that is widely covered by \cite{CarmonaDelarue2017book2}*{Section 4.3}.

\subsection{Stratonovich-It\^o correspondence for measure-dependent integrands}

We are now ready to introduce the main result of the manuscript.

\begin{proposition}
\label{prop:strat-to-ito-main-result}

For almost all $\omega^0 \in \Omega^0$ take $(\mu_t)_{t \in [0,T]} := \big(Law(Y_t(\omega_0,\cdot))\big)_{t \in [0,T]},$ with $Y = (Y_1,\dots Y_d)$ solution to \eqref{eq:GenericYprocess-PartialFlow22} under Assumption \ref{Assump:SDE-Y-mu-2BM}.

Then for all $i = 1,\dots, d$ and $t \in [0,T]$ the Stratonovich SDE for $Y_i$ transforms $\bP$-a.s. to an It\^o SDE according to

\begin{align}
\label{eq:stratonovich-to-ito-main-result}
    \nonumber
    Y_{i,t} &= Y_{i,0} + \int_0^t b_i(s,\mu_s) ~ \dd s + \sum_{j = 1}^m \int_0^t \sigma_{ij}^0(s,\mu_s) ~\dd W^0_{j,s} + \sum_{j = 1}^m \int_0^t \sigma_{ij}^1(s,\mu_s) ~\dd W^1_{j,s} \\
    \Leftrightarrow \quad 
    Y_{i,t} &= Y_{i,0} + \int_0^t b_i(s,\mu_s) ~ \dd t + \sum_{j = 1}^m \int_0^t \sigma^0_{ij}(s,\mu_s) \circ \dd W^0_{j,s} + \sum_{j = 1}^m \int_0^t \sigma^1_{ij}(s,\mu_t) \circ \dd W^1_{j,s} \\
    \nonumber
    &\qquad - \frac12 \sum_{k=1}^{d}\sum_{j=1}^{m} \int_0^t \bE^1\Big[(\partial_\mu \sigma^0_{ij})_k(s,\mu_s,Y^1_s) \Big] \sigma^0_{kj}(s,\mu_s)~\dd s,
\end{align}
where $Y^1= Y^1(\omega_0,\cdot)$ is an independent copy process of $Y$ satisfying \eqref{eq:generic-particle-process-measure-only} for $\bP^0$-almost any $\omega^0 \in \Omega^0$.

\end{proposition}

\begin{remark}
    Equation \eqref{eq:stratonovich-to-ito-main-result} provides a Stratonovich-to-It\^o stochastic integral correspondence in $1$-dimensional component of $Y$.
    Under relevant structural changes to the regularity involved in Assumption \ref{Assump:SDE-Y-mu-2BM} the general form of this rule in multidimensional case will be $\bP$-a.s. given by
    \begin{align}
        \nonumber
        Y_t &= Y_0 + \int_0^t b(s,\mu_s) ~ \dd s + \int_0^t \sigma^0(s,\mu_s) ~ \dd W^0_s + \int_0^t \sigma^1(s,\mu_s) ~ \dd W^1_s
        \\
        \label{eq:stratonovich-to-ito-main-result-tensor}
        \Leftrightarrow \quad Y_t &= Y_0 + \int_0^t b(s,\mu_s) ~ \dd s + \int_0^t \sigma^0(s,\mu_s) \circ \dd W^0_s + \int_0^t \sigma^1(s,\mu_s) \circ \dd W^1_s \\
        \nonumber
        &\qquad - \frac12 \int_0^t \bE^1\Big[\big(\partial_\mu (\sigma^0)_{i}^j\big)_{k}(s,\mu_s,Y^1_s) \Big] \otimes (\sigma^0)^{k}_{j}(s,\mu_s)~\dd s,
    \end{align}
where we use the common Einstein notation for summation across alternating indices.
\end{remark}

\begin{proof}[Proof of Proposition \ref{prop:strat-to-ito-main-result}]
    
    \emph{Step 1. Mollification and compactification}
    We carry out mollification in two steps - firstly we construct the mollifying sequence and later show its convergence.
    As in the \cite{Platonov2019ito}*{Theorem 3.4}, we pick a smooth function $\rho : \bR^d \to \bR^d$ with compact support, letting for any $t \in [0,T]$, $(\sigma^0\star \rho)_t(\mu) := \sigma^0_t(\mu \circ \rho^{-1})$ and for any $t \in [0,T]$ having $\sigma^0~\bP$-a.s. bounded and continuous at every pair $(t,\mu) \in \cP_2(\bR^d),~\partial_\mu \sigma^0 ~ \bP$-a.s. bounded and continuous at every triple $(t,\mu,v)$ for $v \in \Supp (\mu)$ and what follows from local boundedness of $\sigma^0,\sigma^1$ and $\partial_\mu \sigma^0$.
    Now picking the sequence $(\rho_n)_{n \geqslant 1}$ in a way that $(\rho_n, \partial_x\rho_n, \partial^2_{xx}\rho_n)(x) \to (x, I_d,0)$ as $n\to \infty$, we can conclude that $(\sigma^0\star \rho_n)_t(\mu),~(\sigma^1\star \rho_n)_t(\mu),~\partial_\mu (\sigma^0 \star \rho_n)_t(\mu,v)$ converge $\bP$-a.s. to $\sigma^0_t(\mu),~\sigma^1_t(\mu),~\partial_\mu \sigma^0_t (\mu,v)$ respectively. Thus we can assume $\sigma^0, \sigma^1$ and its $\partial_\mu \sigma^0$ $\bP$-a.s. bounded. 

    Again as in \cite{Platonov2019ito}*{Theorem 3.4} we consider $\mu \mapsto (\sigma^0 \star \rho)(\mu * \phi_G)$ instead of $\mu \mapsto (\sigma^0 \star \rho)(\mu)$ with $\phi_G$ - density of standard $d$-dimensional Gaussian distribution $N(0, I_d)$ on $\bR^d$ and $(\mu * \phi_G)(x):=\int_{\bR^d} \phi_G(x-y) d\mu(y)$. Now the support of $\mu *\phi_G$ is the whole $\bR^d$ and $\partial_\mu \sigma^0$ is $\bP$-a.s. continuous at every triple $(t,\mu \circ \phi_G, v), t \in [0,T], v \in \bR^d$. Installing $\phi_{\varepsilon,G}$ - $d$-dimensional Gaussian distribution $N(0,\varepsilon I_d)$ and letting $\varepsilon \searrow 0$, we conclude the $\bP$-a.s. convergence of $\partial_\mu \sigma^0_t (\mu*\phi_{\varepsilon,G},v)$ to $\partial_\mu \sigma^0_t (\mu,v)$ respectively.
    Thus we can assume $\bP$-a.s. uniform continuity of measure expansion terms for the whole $\bR^d$.
    
    Now we are to show that mollification procedure is well-posed. 
    It is straightforward to verify that $\sigma^0_n:= \sigma^0 \star \rho_n$ satisfies $\bP$-a.s. \eqref{cond:integrability-measure-only} uniformly in $n \geqslant 1$.
    Applying twice the dominated convergence theorem we conclude the $\bP$-a.s. convergence for all the $\dd t$ terms. To handle the convergence of the stochastic integrals one additionally requires an argument across the quadratic variation, as written in \cite{Platonov2019ito}*{Theorem 2.3} and localisation.
    We copy the procedure above to conclude that $\sigma^0,\sigma^1,\partial_\mu\sigma^1~\bP$-a.s. have compact support.

    \emph{Step 2. Approximation.} By our mollification argument one can assume the $\sigma^0,~ \sigma^1, \partial_\mu \sigma^0$ to be $\bP$-a.s. bounded and $\bP$-a.s. uniformly continuous in respective topology spaces. Furthermore, one can assume $Y_t$ to be $\bP$-a.s. bounded, as its distribution has compact support.
    We construct twin processes $(Y^l_t)_{t \in [0,T]}, ~ l = 1,\dots, N$ of $(Y_t)_{t \in [0,T]}$ each supporting its own independent $m$-dimensional Brownian motion $(W^{1,l}_t)_{t\in[0,T]}$ that generate $(\Omega^{1,l},\cF^{1,l}, \bF^{1,l},\bP^{1,l})$ alongside with $\cF_0^l$, altogether forming a copy of $(\Omega^{1},\cF^{1}, \bF^{1},\bP^{1})$. Since the stochastic basis $(\Omega,\cF,\bF,\bP)$ of our initial space is constructed as a completion of $(\Omega^0 \times \Omega^1, \cF^0 \otimes \cF^1, \bF^0 \otimes \bF^1, \bP^0 \otimes \bP^1)$ augmented in a right-continuous way and then completed, we introduce a new product basis $(\Omega^{l},\cF^{l}, \bF^{l},\bP^{l}) $ to be completion of $(\Omega^0 \times \Omega^{1,l}, \cF^0 \otimes \cF^{1,l}, \bF^0 \otimes \bF^{1,l}, \bP^0 \otimes \bP^{1,l})$ augmented in right-continuous way and then completed. Now we copy the dynamics of $(Y_t)_{t \in [0,T]}$, as 
    \begin{align}
    \label{eq:generic-particle-process-measure-only}
        \dd Y_t^l = b(t,\mu_t^l)~\dd t + \sigma(t,\mu_t^l) ~ \dd W_t^0 + \sigma(t,\mu_t^l) ~ \dd W^{1,l}_t, \quad Y^l_0 = Y_0^l,
    \end{align}
    where $Y_0^l$ are $\bP^0$ i.i.d. copies of $Y^0$.
The components of $Y^l= (Y^l_1,\dots, Y^l_d) \in \bR^d$ satisfy the following 1-dimensional SDE  
    \begin{align}
    \label{eq:generic-particle-process-measure-only-1d}
        \dd Y_{i,t}^l = b_i(t,\mu_t^l) ~ \dd t + \sum_{j = 1}^m \sigma_{ij}(t,\mu_t^l) ~ \dd W_{j,t}^0 + \sum_{j = 1}^m \sigma_{ij}(t,\mu_t^l) ~ \dd W^{1,l}_{j,t}, \quad Y^l_0 = Y_0^l,
    \end{align}
    for all $i \in 1,\dots,d$.
    Now we construct a total stochastic basis $(\Omega^{1,\dots,N},\cF^{1,\dots,N}, \bF^{1,\dots,N},\bP^{1,\dots,N})$, where
    \begin{align*}
        \Omega^{1,\dots, N} = \Omega^0 \times \Omega^1 \times \prod_{l=1}^N \Omega^{1,l}&, \quad  \cF^{1,\dots, N} = \cF^0 \otimes \cF^1 \otimes \bigotimes_{l=1}^N \cF^{1,l}, \\
        \bF^{1,\dots, N} = \bF^0 \otimes \bF^1 \otimes \bigotimes_{l=1}^N \bF^{1,l}&, \quad 
        \bP^{1,\dots, N} = \bP^0 \otimes  \bP^1 \otimes \bigotimes_{l=1}^N \bP^{1,l},
    \end{align*}
    where we again and finally augment the filtration in a right-continuous way and complete.
    We underline that processes $(Y^l_t)(\omega^0,\cdot),~ l = 1, \dots N$ are i.i.d. $\bP^0$-a.s.
    
    Hereinafter while fixing the $\omega^0\in \Omega^0$, and for the sake of simplicity we will omit adding the $(\omega^0,\cdot)$ to the processes $Y_t$ and its copies to highlight the respective relation to $\omega^0$, but will leave in after $\bar\mu^N_t$ as to underline the nature of this dependency.
    
    We denote the flow of marginals for almost all $\omega^0 \in \Omega^0$ as $\bar \mu_t^N(\omega^0,\cdot) := \frac1N\sum_{l = 1}^N \delta_{Y_t^l(\omega^0,\cdot)}$ for $t \in [0,T]$ and the empirical projection of $\sigma^0, \sigma^1$ as $\sigma^{0,N}, \sigma^{1,N}$. 

    We consider the components of the vector $Y$ given by \eqref{eq:GenericYprocess-PartialFlow22-1d} and the components of its copies \eqref{eq:generic-particle-process-measure-only} given by \eqref{eq:generic-particle-process-measure-only-1d}.

Now the standard Stratonovich-It\^o relation \cite{kunita1997stochastic}*{Theorem 2.3.5.} for stochastic integrals applied to $\sigma^{0,N}_{ij},~\sigma^{1,N}_{ij},~\bP$-a.s satisfies for all $i = 1,\dots, d$, 
    \begin{align}
        \nonumber
        & \sum_{j = 1}^m\int_0^t\sigma^{0,N}_{ij}(s,Y^1_s,\dots, Y^N_s) ~ \dd W^0_{j,s} +  \sum_{j = 1}^m\int_0^t\sigma^{1,N}_{ij}(s,Y^1_s,\dots, Y^N_s) ~ \dd W^1_{j,s} \\
        \label{eq:strat-ito-relation-after-applying-kunita}
        =  & \sum_{j = 1}^m\int_0^t \sigma^{0,N}_{ij}(s,Y^1_s,\dots, Y^N_s) \circ \dd W^0_{j,s} + \sum_{j = 1}^m\int_0^t \sigma^{1,N}_{ij}(s,Y^1_s,\dots, Y^N_s) \circ \dd W^1_{j,s}\\
        \nonumber
        &\qquad - \frac1{2}\sum_{j = 1}^m\int_0^t \dd \big\langle\sigma^{0,N}_{ij} (s,Y^1_s,\dots,Y^N_s), W_{j,s}^0\big\rangle
        - \frac1{2} \sum_{j = 1}^m\int_0^t \dd \big\langle\sigma^{1,N}_{ij} (s,Y^1_s,\dots,Y^N_s), W_{j,s}^1\big\rangle.
    \end{align}

    Using the conditional expectations one can show that mutual independence of the covariates we argue that the last term of \eqref{eq:strat-ito-relation-after-applying-kunita} is $0 ~\bP$-a.s..

    We transform further as in \cite{kloedenplaten1992NumericalSDEs}*{Section 4.9.} to have $\bP$-a.s. for all $i = 1,\dots, d$, 
    \begin{align}
        \nonumber
        & \sum_{j = 1}^m\int_0^t\sigma^{0,N}_{ij}(s,Y^1_s,\dots, Y^N_s) ~ \dd W^0_{j,s} +  \sum_{j = 1}^m\int_0^t\sigma^{1,N}_{ij}(s,Y^1_s,\dots, Y^N_s) ~ \dd W^1_{j,s} 
        \\
        \label{eq:strat-ito-relation-particles}
         & = \sum_{j = 1}^m\int_0^t \sigma^{0,N}_{ij}(s,Y^1_s,\dots, Y^N_s) \circ \dd W^0_{j,s} + \sum_{j = 1}^m \int_0^t \sigma^{1,N}_{ij}(s,Y^1_s,\dots, Y^N_s) \circ \dd W^1_{j,s} \\
        \nonumber
        &\qquad - \frac1{2} \sum_{j = 1}^m \sum_{l=1}^{N} \int_0^t \sum_{k = 1}^d  \partial_{y^l_k}\sigma^{0,N}_{ij}(s,Y^1_s,\dots,Y^N_s) ~\dd\big\langle Y_{k,s}^l,W_{j,s}^0\big\rangle.
    \end{align}

    Going back to $\sigma^0$ and $\sigma^1$, transforming the first bracket according to \eqref{eq:generic-particle-process-measure-only} since the dynamics of the copy processes is independent of $W^1$, we write that $\bP$-a.s. for all $i = 1,\dots, d$,
    \begin{align}
        \nonumber
        & \sum_{j = 1}^m \int_0^t\sigma_{ij}^{0}(\bar \mu_s^N(\omega^0,\cdot)) ~ \dd W^0_{j,s} + \sum_{j = 1}^m \int_0^t \sigma_{ij}^{1}(\bar \mu_s^N(\omega^0,\cdot)) ~ \dd W^1_{j,s} \\
        \label{eq:strat-ito-relation-empirical-measure}
        & =  \sum_{j = 1}^m \int_0^t\sigma_{ij}^{0}(\bar \mu_s^N(\omega^0,\cdot)) \circ \dd W^0_{j,s} + \sum_{j = 1}^m \int_0^t \sigma_{ij}^{1}(\bar \mu_s^N(\omega^0,\cdot)) \circ \dd W^1_{j,s} \\
        \nonumber
        &\qquad - \sum_{k=1}^d \sum_{j = 1}^m  \int_0^t\frac1{2N} \sum_{l=1}^{N}(\partial_\mu \sigma^0_{ij})_k(s,\bar \mu_s^N(\omega^0,\cdot), Y^l_s) ~\sigma^0_{kj}(s,\bar \mu_s^N(\omega^0,\cdot)) ~\dd s.
    \end{align}
    Taking conditional expectations on the above formula $\bE^{1,1,\dots, N}\big[ \cdot\big] :=\bE^{\bP^{1,\dots,N}} \big[ \cdot |\: \cF^0 \otimes \cF^1\big]$ we have $\bP$-a.s. by the stochastic Fubini theorem (see \cite{Veraar2010FubiniRevisited}*{Theorem 3.5}) for all $i = 1,\dots, d$,
    \begin{align}
        \nonumber
        &\sum_{j = 1}^m \int_0^t\bE^{1,1,\dots, N}\Big[\sigma_{ij}^{0}(s,\bar \mu_s^N(\omega^0,\cdot)) ~ \dd W^0_{j,s} \Big] + \sum_{j = 1}^m \int_0^t \bE^{1,1,\dots, N} \Big[\sigma_{ij}^{1}(s,\bar \mu_s^N(\omega^0,\cdot)) ~ \dd W^1_{j,s} \Big] 
        \\
        \label{eq:strat-ito-relation-empirical-measure-with-expectation}
        &=\sum_{j = 1}^m \int_0^t\bE^{1,1,\dots, N}\Big[\sigma_{ij}^{0}(s,\bar \mu_s^N(\omega^0,\cdot)) \circ \dd W^0_{j,s} \Big] +  \sum_{j = 1}^m \int_0^t\bE^{1,1,\dots, N}\Big[\sigma_{ij}^{1}(s,\bar \mu_s^N(\omega^0,\cdot)) \circ \dd W^1_{j,s} \Big] \\
        \nonumber
        & \qquad - \sum_{k=1}^d \sum_{j = 1}^m \int_0^t \frac1{2}\bE^{1,1,\dots, N} \Big[ (\partial_{\mu}\sigma_{ij}^{0})_k(s,\bar \mu_s^N(\omega^0,\cdot), Y^1_s) ~ \sigma_{kj}^{0}(\bar \mu_s^N(\omega^0,\cdot))\Big] \dd s.
    \end{align}
    Letting $n \to \infty$, \eqref{eq:strat-ito-relation-empirical-measure-with-expectation} converges to \eqref{eq:stratonovich-to-ito-main-result} $\bP$ a.s. by the dominated convergence theorem for stochastic integrals (applying localisation when necessary) and continuity of $\sigma^0, \sigma^1$ and $\partial_\mu\sigma^0$ (see \cite{Platonov2019ito}*{Theorem 2.3} for detailed arguments).
\end{proof}

We are able to establish a more general but less transparent result under weaker conditions.
\begin{proposition}
\label{prop:stratonovich-to-ito-general-result}
For almost all $\omega^0 \in \Omega^0$ take $(\mu_t)_{t \in [0,T]} := \big(Law(Y_t(\omega_0,\cdot))\big)_{t \in [0,T]},$ with $Y = (Y_1,\dots Y_d)$ solution to \eqref{eq:GenericYprocess-PartialFlow22} under Assumption \ref{Assump:SDE-Y-mu-2BM}, where we exclude all conditions on the existence and regularity of $\partial_\mu \sigma_{ij}^0$.

Then for all $i = 1,\dots d$ and $t \in [0,T]$ the Stratonovich SDE for $Y_i$ transforms $\bP$-a.s. to It\^o SDE according to

\begin{align}
\label{eq:stratonovich-to-ito-general-result}
    \nonumber
    Y_{i,t} &= Y_{i,0} + \int_0^t b_i(s,\mu_s) ~\dd s + \sum_{j = 1}^m \int_0^t \sigma_{ij}^0(s,\mu_s) ~\dd W^0_{j,s} + \sum_{j = 1}^m \int_0^t \sigma_{ij}^1(s,\mu_s) ~\dd W^1_{j,s} \\
    \Leftrightarrow \quad 
    \nonumber
    Y_{i,t} &= Y_{i,0} + \int_0^t b_i(s,\mu_s) ~\dd s + \sum_{j = 1}^m \int_0^t \sigma_{ij}^0(s,\mu_s) \circ \dd W^0_{j,s} + \sum_{j = 1}^m \int_0^t \sigma_{ij}^1(s,\mu_s) \circ \dd W^1_{j,s} \\
    & \qquad - \frac1{2} \sum_{j = 1}^m \big\langle\sigma^{0,N}_{ij} (t,\mu_t), W_{j,t}^0\big\rangle.
\end{align}

\end{proposition}

\begin{proof}[Proof of Proposition \ref{prop:stratonovich-to-ito-general-result}]
We follow the same arguments as in a proof of \eqref{prop:strat-to-ito-main-result} up to establishing equation \eqref{eq:strat-ito-relation-after-applying-kunita}.
Again noticing that the last term is $\bP$-a.s. $0$, we pass to the limit as $N \to \infty$ in \eqref{eq:strat-ito-relation-after-applying-kunita} yo obtain \eqref{eq:stratonovich-to-ito-general-result} by dominated convergence theorem, continuity of $\sigma^0$ and $\sigma^1$ and the limit property of cross-variation (see \cite{kunita1981some}*{Corollary 2.2.16}).
\end{proof}

In case of $\sigma^0$ being a progressively-measurable random process, the last term of the Stratonovich form of \eqref{eq:stratonovich-to-ito-general-result} will capture not only the cross-variance between $\sigma^0(t,\mu_\cdot)$ and $W^0_\cdot$ but also that between $\sigma^0(\cdot,\mu_t)$ and $W^0_\cdot$. Nonetheless, without the differentiability assumption on $\sigma^0$ one cannot expand the term further and recover Proposition \ref{prop:strat-to-ito-main-result}.

\subsection{Stratonovich-It\^o correspondence for McKean-Vlasov SDEs}

Enriching the drift and diffusion functions by adding explicit dependence on the process itself, we write the equivalence between formulation under the framework of general McKean-Vlasov SDEs.

Take a measurable $(b,\sigma^0, \sigma^1): [0,T] \times \bR^d \times \cP_2(\bR^d) \to \bR^d \times \bR^{d \times m}\times \bR^{d \times m}$ and define the following process
\begin{align}
\label{eq:GenericYprocess-PartialFlow-with-Yprocess}
    \dd Y_t = b(t,Y_t,\mu_t) ~ \dd t + \sigma^0(t,Y_t,\mu_t) ~ \dd W^0_t + \sigma^1(t,Y_t,\mu_t) ~ \dd W^1_t, \quad Y_0\in L^2(\Omega,\cF_0,\bP),
\end{align}
and $\mu_t := \Law(Y_t(\omega_0,\cdot))$ for $\bP^0$-almost any $\omega_0$.
The components of $Y = (Y_1,\dots, Y_d) \in \bR^d$, satisfy the following 1-dimensional SDE  for all $i = 1,\dots, d$, 
\begin{align*}
    \dd Y_{i,t} = b_i(t,Y_t,\mu_t) ~ \dd t + \sum_{j = 1}^m \sigma_{ij}^0 (t,Y_t,\mu_t) ~ \dd W^0_{j,t} +
    \sum_{j = 1}^m \sigma_{ij}^1(t, Y_t,\mu_t) ~ \dd W^1_{j,t}, \:\: Y_{0,i}\in L^2(\Omega,\cF_0,\bP).
\end{align*}

Moreover, the involved coefficients satisfy the next condition.
\begin{assumption}
\label{Assump:SDE-Y-mu-2BM-full-process}
Let $Y_0 \in L^2(\Omega,\cF_0,\bP)$ ($Y_0$ is $\cF_0$-measurable and independent of $W_t^0,~W_t^1, ~ t \in [0,T]$). Take $b: [0,T] \times \bR^d \times \cP_2(\bR^d) \to \bR^d$ and $\sigma^0,~\sigma^1:[0,T] \times \bR^d \times \cP_2(\bR^d) \to \bR^{d \times m}$ such that it satisfies
\begin{enumerate}[i)]
    
    \item For any $(t,y) \in [0,T]\times \bR^d$, the maps $\mu \mapsto \sigma^0(t,y,\mu),~ \mu \mapsto \sigma^1(t,y,\mu)$ are continuous in topology induced by the Wasserstein metric for any $\mu \in \cP_2(\bR^d)$;

    \item For any $(t,y) \in [0,T]\times \bR^d,~i = 1,\dots, d$ and $j = 1,\dots, m$, the map $\mu \mapsto (\sigma^0)_{ij}(t,y,\mu)$ continuously L-differentiable $\bP$-a.s. at every point $\mu \in \cP_2(\bR^d)$. Moreover, for all $i = 1,\dots, d$ and $j = 1,\dots, m,~\partial_\mu \sigma_{ij}^0: [0,T] \times \cP_2(\bR^d)$ is $\bP$-a.s. joint-continuous and locally bounded at every quadruple $(t,y,\mu,v)$, with $(t,y,\mu) \in [0,T] \times \bR^d \times \cP_2(\bR^d),~ v \in \Supp(\mu)$;

    \item For any $(t,\mu) \in [0,T] \times \cP_2(\bR^d)$, the maps $y \mapsto \sigma^0(t,y,\mu),~y \mapsto \sigma^1(t,y,\mu)$ are $\cC^1(\bR^d),~\bP$-a.s. at every $y \in \bR^d$, with $\partial_y \sigma^0, \partial_{y} \sigma^1$ being $\bP$-a.s.~joint continuous at every triple $(t,y,\mu) \in [0,T] \times \bR^d \times \cP_2(\bR^d), ~\bP$-a.s.;
    \item The involved coefficients satisfy for any compact $K \subset \bR^d \times \cP_2(\bR^d)$
        \begin{align}
            \nonumber
            \label{cond:integrability-full-process}
            \sup_{(y,\mu) \in K} \sup_{i \in \{1,\dots,d \}} &\int_0^T \Big[|b_i(s,y,\mu)| + |\sigma_i^0(s,y,\mu)|^2 + |\sigma_i^1(s,y,\mu)|^2 
            + |\partial_y \sigma_i^0(s,y,\mu)|^2 \\
            &\qquad + |\partial_y \sigma_i^1(s,y,\mu)|^2 
            + \Big|\int_{\bR^d}\partial_\mu \sigma_i^0(s,y,\mu,v) \mu(\dd v)\Big|^2\Big]\dd s < \infty.
        \end{align}
\end{enumerate}
\end{assumption}

Following the arguments of Proposition \ref{prop:strat-to-ito-main-result} with the classical result of \cite{kunita1981some,kloedenplaten1992NumericalSDEs} one can prove the following result.

\begin{theorem}
    For almost all $\omega^0 \in \Omega^0$ take $(\mu_t)_{t \in [0,T]} := \big(Law(Y_t(\omega_0,\cdot))\big)_{t \in [0,T]},$ with $Y$ solution to \eqref{eq:GenericYprocess-PartialFlow-with-Yprocess} under Assumption \ref{Assump:SDE-Y-mu-2BM-full-process}.
    
    Then for all $i = 1,\dots, d$ and $t \in [0,T]$ the Stratonovich SDE for $Y_i$ transforms $\bP$-a.s. to It\^o SDE according to
    \begin{align}
        \nonumber
        Y_{i,t} &= Y_{i,0} + \int_0^t b_i(s,Y_s,\mu_s) ~\dd s + \sum_{j = 1}^m \int_0^t \sigma_{ij}^0(s,Y_s,\mu_s) ~ \dd W^0_{j,s} + \sum_{j = 1}^m \int_0^t \sigma_{ij}^1(s,Y_s,\mu_s) ~ \dd W^1_{j,s} \\
        \nonumber
        \Leftrightarrow \quad 
        Y_{i,t} &= Y_{i,0} + \int_0^t b_i(s,Y_s,\mu_s) ~\dd s + \sum_{j = 1}^m \int_0^t \sigma^0_{ij}(s,Y_s,\mu_s) \circ \dd W^0_{j,s} + \sum_{j = 1}^m \int_0^t \sigma^1_{ij}(s,Y_s,\mu_s) \circ \dd W^1_{j,s} \\
        \label{eq:stratonovich-to-ito-main-result-full-case}
        &\qquad - \frac12 \sum_{k=1}^{d}\sum_{j=1}^{m} \int_0^t \bE^1\Big[(\partial_\mu \sigma^0_{ij})_k(s,Y_s,\mu_s,Y^1_s) \Big] \sigma^0_{jk}(s,Y_s,\mu_s)~\dd s\\
        \nonumber
        &\qquad - \frac12 \sum_{k=1}^{d}\sum_{j=1}^{m} \int_0^t \partial_{y_k} \sigma_{ij}^0(s,Y_s,\mu_s)~\sigma_{jk}^0(s,Y_s,\mu_s)~\dd s \\
        \nonumber
        &\qquad - \frac12 \sum_{k=1}^{d}\sum_{j=1}^{m} \int_0^t \partial_{y_k} \sigma_{ij}^1(s,Y_s,\mu_s) ~ \sigma_{jk}^1(s,Y_s,\mu_s)~\dd s,
    \end{align}
    where $Y^1= Y^1(\omega_0,\cdot)$ is an independent copy process of $Y$ satisfying \eqref{eq:GenericYprocess-PartialFlow-with-Yprocess} for $\bP^0$-almost any $\omega^0 \in \Omega^0$.

\end{theorem}

%
%


\end{document}